\documentclass[11pt]{amsart}
\usepackage{amsmath}
  \usepackage{url,graphicx}
  \usepackage{amssymb, color, pstricks-add, manfnt}
  \usepackage{amsmath, amsthm,  amsfonts}
  \usepackage[plainpages=false]{hyperref}

  \theoremstyle{plain}
  \newtheorem{theorem}{Theorem}[section]
  \newtheorem{lemma}{Lemma}[section]

   \theoremstyle{remark}
  \newtheorem{remark}{Remark}[section]

  \numberwithin{equation}{section}
  \numberwithin{figure}{section}
\renewcommand{\baselinestretch}{1.00}
\begin{document}

\title{On the second boundary value problem for Monge-Amp\`ere type equations and geometric optics}

\author{Feida Jiang}
\address{College of Mathematics and Statistics, Nanjing University of Information Science and Technology, Nanjing 210044, P. R. China}
\email{jfd2001@163.com}

\author{Neil S. Trudinger}
\address{Centre for Mathematics and Its Applications, The Australian National University,
              Canberra ACT 0200, Australia}
\address{School of Mathematics and Applied Statistics, University of Wollongong, Wollongong, NSW 2522, Australia}
\email{Neil.Trudinger@anu.edu.au; neilt@uow.edu.au}

  %  General info

%\date{\today}


\maketitle

\abstract {In this paper, we prove the existence of classical solutions to second boundary value problems for generated prescribed Jacobian equations, as recently developed by the second author, thereby obtaining  extensions of classical solvability of optimal transportation problems to problems arising in near field geometric optics.  Our results depend in particular on  {\it a priori} second derivative estimates recently established by the authors under weak co-dimension one convexity hypotheses on the associated matrix functions with respect to the gradient variables, (A3w). We also avoid domain deformations by using the convexity theory of generating functions to construct unique initial solutions for our
homotopy family, thereby enabling application of the degree theory for nonlinear oblique boundary value problems.}
\endabstract

%\footnote { }

\baselineskip=12.8pt
\parskip=3pt
\renewcommand{\baselinestretch}{1.38}

\section{Introduction}\label{Section 1}

\vskip10pt
Let $\Omega$ be a domain in $n$ dimensional Euclidean space $\mathbb{R}^n$, and $Y$ be a mapping from $\Omega \times \mathbb{R}\times \mathbb{R}^n$ into $\mathbb{R}^n$. The prescribed Jacobian equation (PJE) has the following form,
\begin{equation}\label{PJE}
\det DY(\cdot, u, Du)=\psi(\cdot, u, Du),
\end{equation}
where $\psi$ is a given scalar function on $\Omega \times \mathbb{R}\times \mathbb{R}^n$ and $Du$ is the gradient vector of the function $u: \Omega\rightarrow \mathbb{R}$. We are concerned here with mappings $Y$ which can be generated by a smooth generating function $g$ defined on domains $\Gamma \subset \mathbb{R}^n\times \mathbb{R}^n \times \mathbb{R}$, which embrace applications in geometric optics and optimal transportation \cite{JT-Pogorelov, T2014}. In the general set up, we assume $g\in C^4(\Gamma)$, where $\Gamma$ has the property that the projections
$$I(x,y) = \{z\in\mathbb{R} |\  (x,y,z)\in\Gamma\}$$

\noindent are open intervals.
Denoting
\begin{equation}\label{U}
\mathcal{U}=\{(x,g(x,y,z),g_x(x,y,z))|\  (x,y,z)\in \Gamma\},
\end{equation}
then we have the following conditions,
\begin{itemize}
\item[{\bf A1}:]
For each $(x,u,p) \in \mathcal U$, there exists a unique point $(x,y,z)\in\Gamma$
satisfying
$$g(x,y,z) = u, \ \ g_x(x,y,z) = p.$$
\item[{\bf A2}:]
$g_z < 0$,  $\det E \ne 0$, in $ \Gamma$, where $E$ is the $n\times n$ matrix given by
$$E = [E_{i,j}] =  g_{x,y} - (g_z)^{-1}g_{x,z}\otimes g_y.$$
\end{itemize}
The sign of $g_z$ in A2 can be changed as we wish. Here we fix the sign of $g_z$ to be negative in accordance with \cite{JT-Pogorelov, T2014}. By defining $Y(x,u,p)=y$ and $Z(x,u,p)=z$ in A1, the mapping $Y$ together with the dual function $Z$ are generated by equations
\begin{equation}\label{generating equation}
g(x,Y,Z)=u, \quad g_x(x,Y,Z)=p.
\end{equation}
Since the Jacobian determinant of the mapping $(y,z)\rightarrow (g_x,g)(x,y,z)$ is $g_z\det E, \neq 0$ by A2, the functions $Y$ and $Z$ are $C^3$  smooth. By differentiating \eqref{generating equation} with respect to $p$, we have $Y_p=E^{-1}$. Also, by differentiating \eqref{generating equation} for $p=Du$, with respect to $x$, we obtain the generated prescribed Jacobian equation (GPJE), 
\begin{equation}\label{GPJE}
\mathcal F[u]:=\det[D^2u-g_{xx}(\cdot,Y(\cdot,u,Du),Z(\cdot,u,Du))]=\det E(\cdot,Y,Z)\psi(\cdot,u,Du),
\end{equation}
when the one-jet $J_1[u](\Omega):=\{(x,u,Du) |\ x\in \Omega \}\subset \mathcal U$, which can also be calculated from equation \eqref{PJE} directly. As usual, we shall denote 
\begin{equation}\label{A,B}
A(\cdot,u,p)=g_{xx}(\cdot,Y(\cdot,u,p),Z(\cdot,u,p)), \  B(\cdotp,u,p)=\det E(\cdot,Y(\cdot,u,p),Z(\cdot,u,p))\psi(\cdot,u,p). 
\end{equation}
Then a function $u\in C^2(\Omega)$ is elliptic (degenerate elliptic) for equation \eqref{GPJE}, whenever $D^2u-A(\cdot,u,Du)>0$, ($\ge 0$), which implies the right hand side $B(\cdot,u,Du)>0$, ($\ge 0$). We refer the reader to \cite{T2014} for more background material about generated prescribed Jacobian equations.

The second boundary value problem for equation \eqref{PJE} is to prescribe the image
\begin{equation}\label{Second boundary data}
Tu(\Omega):=Y(\cdot,u,Du)(\Omega)=\Omega^*,
\end{equation}
where $\Omega^*\subset\mathbb{R}^n$ is a target domain. For applications to  geometric optics, the function $\psi$ is separable in the sense that 
\begin{equation}\label{psi}
|\psi|(x,u,p)=\frac{f(x)}{f^*\circ Y(x,u,p)},
\end{equation}
for positive intensities $f\in L^1(\Omega)$ and $f^*\in L^1(\Omega^*)$. Then a necessary condition for the existence of an elliptic solution with the mapping $Tu$ being a diffeomorphism, to the second boundary value problem \eqref{GPJE}, \eqref{Second boundary data}, is the conservation of energy
\begin{equation}\label{conservation of energy}
\int_\Omega f =\int_{\Omega^*} f^*.
\end{equation}
We shall assume $f$ and $f^*$ are both smooth and have positive lower bounds and upper bounds. Note that in optimal transportation \cite{MTW2005, TW2009}, $f$ and $f^*$ are densities, and the condition \eqref{conservation of energy} is called the mass balance condition.

The strict monotonicity property of the generating function $g$ with respect to $z$, enables us to define a dual generating function $g^*$,
\begin{equation}\label{dual generating function g*}
g(x,y,g^*(x,y,u))=u,
\end{equation}
with $(x,y,u)\in \Gamma^* :=\{(x,y,g(x,y,z)) |  (x,y,z) \in \Gamma\}$, $g^*_x =-g_x/g_z$, $g^*_y=-g_y/g_z$ and $g^*_u=1/g_z$, which leads to a dual condition to A1, namely

\begin{itemize}
\item[{\bf A1*}:]
The mapping $Q: = -g_y/g_z$ is one-to-one in $x$, for all $(x,y,z) \in \Gamma$.
\end{itemize}
Note that the Jacobian matrix of the mapping $x\to Q(x,y,z)$ is $-E^t/g_z$ where $E^t$ is the transpose of $E$ so its determinant will not vanish when condition A2 holds,  that is A2 is self dual.

We assume also the following conditions on the generating function $g$ which are expressed in terms of  the matrix $A$. Extending the necessary assumption A3w for regularity in optimal transportation in \cite{Loeper09Acta, Tru2006, TW2009}, we assume the following regular condition for the matrix function $A$ with respect to $p$, which we formulate together with its strict version \cite{MTW2005}.
\begin{itemize}
\item[{\bf A3w}] ({\bf A3}):
The matrix function $A$ is regular (strictly regular) in $\mathcal U$, that is $A$ is co-dimension one convex (strictly co-dimension one convex) with respect to $p$
in the sense that,
$$A^{kl}_{ij}\xi_i\xi_j\eta_k\eta_l: = (D_{p_kp_l}A_{ij}) \xi_i\xi_j\eta_k\eta_l \ge 0, (>0) $$

\vspace {0.2cm}

\noindent in $\mathcal U$, for all  $\xi,\eta \in \mathbb{R}^n$ such that $\xi \!\cdot\! \eta = 0$.
\end{itemize}

We also need a monotonicity condition on the matrix $A$ with respect to $u$, namely A4w or A4*w.
\begin{itemize}
\item[{\bf A4w}] ({\bf A4*w}): The matrix $A$ is monotone increasing (decreasing) with respect to $u$ in $\mathcal U$, that is
$$D_uA_{ij}\xi_i\xi_j \ge 0, \ (\le 0)$$

\vspace {0.2cm}

\noindent  in $\mathcal U$, for all  $\xi \in \mathbb{R}^n$.

\end{itemize}

As in \cite{JT-Pogorelov, T2014}, we need to impose convexity assumptions on the set $\mathcal U$ when we apply
conditions A3w and A4w (or A4*w). For globally smooth solutions these conditions will be assured by the necessary convexity assumptions on our domains and we may restrict our consideration accordingly. Thus we need only assume that
the product $\bar\Omega\times\bar\Omega^*$ lies in the projection of $\Gamma$ on
$\mathbb{R}^n\times\mathbb{R}^n$ with $\mathcal U$ replaced by
$$\mathcal U(\Omega,\Omega^*) = \{ (x,g(x,y,z),g_x(x,y,z)) | \  x\in \bar \Omega, y\in \bar \Omega^*, z\in I(x,y) \}$$

\noindent in conditions A3w, A3, A4w and A4*w. Note that there is no loss of generality in maintaining A1, A2 and A1* as $\Gamma$ can be redefined so its projection on $\mathbb{R}^n\times\mathbb{R}^n$ is close to $\overline\Omega\times\overline\Omega^*$.

We next have the following  condition to guarantee the appropriate
controls on $J_1[u]$, which is a refinement of condition G5
in \cite{T2014}; (see also \cite{T2014-1}). Namely, writing $J(x,y)
=g(x,y,\cdot)I(x,y)$,
 we assume:
\begin{itemize}
\item[{\bf A5}:]
There exists  an infinite open interval $J_0$
and a positive constant $K_0$, such that $J_0\subset J(x,y)$ and
$$  |g_x(x,y,z)| < K_0, $$
for all $x\in\bar\Omega, y\in \bar\Omega^*, g(x,y,z) \in J_0$.
\end{itemize}

Note that we can assume that $J_0 = (m_0,\infty)$ for some constant $m_0 \ge -\infty$ or $J_0 = (-\infty, M_0)$, for a constant $M_0$. The situation when $J_0$ is finite will be considered at the end of our existence proof.

Finally to complete our hypotheses we adopt the following domain convexity definitions from \cite {Tru2008,LT2016}, which extend the corresponding conditions for optimal transportation in \cite{TW2009}. It will be convenient to express these more generally in terms of the mapping $Y$ generated by $g$.

The $C^2$ domain $\Omega$ is $Y$-convex (uniformly $Y$-convex) with respect to $\Omega^*\times J$, where $ J$ is an open interval in $J(\Omega,\Omega^*)$, if it is connected and
\begin{equation}\label{Y-convex}
        [D_i\gamma_j(x)- D_{p_k}A_{ij}(x,u,p)\gamma_k(x)]\tau_i\tau_j\ge 0,(\delta_0),
    \end{equation}
for all $x\in\partial\Omega$, $u\in {J}$, $Y(x,u,p) \in \Omega^*$, unit outer normal $\gamma$ and unit tangent vector $\tau$, (for some constant $\delta_0>0$).

The domain $\Omega^*$ is $Y^*$-convex (uniformly $Y^*$-convex) with respect to $\Omega\times  J$
if  the images
$$\mathcal P(x,u,\Omega^*) = \{p\in\mathbb{R}^n | \ (x,u,p)\in \mathcal U, Y(x,u,p)\in\Omega^*\}$$
are convex for all $(x,u)\in \Omega\times J$, (uniformly convex for all $x\in\overline\Omega$, $u\in\overline J$).

The lack of symmetry between our formulations is caused by using the $u$ variable in both cases. We will discuss them further, including their relationship with the $g$-convexity notion introduced in \cite {T2014}, in conjunction with the application of our estimates from \cite{JT-Pogorelov} in Section 3.

We can now state our main theorem for the second boundary value problem \eqref{GPJE}, \eqref{Second boundary data}.

\begin{theorem}\label{Th1.1}
Let $g\in C^4(\Gamma)$ be a generating function satisfying conditions A1, A2, A1*, A3w, A5 and either A4w, A4*w or A3, with $C^4$ bounded domains
$\Omega$, $\Omega^*$ in $\mathbb{R}^n$ which are respectively uniformly $Y$-convex and uniformly $Y^*$-convex with respect to each other and  any interval  $J\subset\subset J_0$. Suppose also the function $\psi$ satisfies \eqref{psi}, \eqref{conservation of energy}. Then there exists an elliptic solution $u\in C^3(\bar \Omega)$ of the second boundary value problem \eqref{GPJE}, \eqref{Second boundary data}, whose range lies in $J_0$. Furthermore, the mapping $Tu$ is a $C^2$ smooth diffeomorphism from $\bar\Omega$ to $\bar \Omega^*$.
\end{theorem}

Note that by varying $m_0$ or $M_0$ we obtain the existence of an infinite number of solutions. Also we note that elliptic solutions of \eqref{GPJE} will  be $g$-convex in the sense of \cite{T2014} under appropriate domain convexity conditions and this property is also crucial in our proof; (see Section \ref{Section 2}).

We remark that the {\it a priori} second order derivative estimates up to the boundary and the existence of classical solutions for the second boundary value problem for (far field) geometric optics problems are raised in \cite{GW1998} in the context of far field reflector antenna problems. Such problems are solved in the broader context of optimal transportation in \cite{TW2009}. In this paper, we consider more general situations of  second boundary value problems for generated prescribed Jacobian equations, which embrace those examples in near field optics problems in \cite{JT-Pogorelov, LT2016, T2014}. Moreover, we can avoid the $c$-boundedness of domains as in \cite{TW2009} or the $Y$-boundedness  as in \cite{Tru2006, LT2016}, since both the second derivative estimates in \cite{JT-Pogorelov} and the continuity method used in Section \ref{Section 3} do not depend on such conditions. Based on the second derivative estimate, Corollary 3.1 in \cite{JT-Oblique II}, we also have the existence of the second boundary value problem for more general augmented Hessian equations; (Remark \ref{Remark 3.4}).

This paper is organised as follows. In Section \ref{Section 2}, we construct a uniformly $g$-convex function which approximately satisfies the boundary condition \eqref{Second boundary data}. The construction is realised by extension and mollification of an initial construction of a uniformly elliptic function and uses the convexity theory of generating functions developed in \cite{T2014, T2014-1}. In Section \ref{Section 3}, we start from the uniformly $g$-convex function constructed in Section \ref{Section 2} to prove the existence result, Theorem \ref{Th1.1}, by using a more elaborate version of the degree argument  employed in \cite{LT2016}, which does not require domain deformation. We also give a more precise version of Theorem \ref{Th1.1}, which permits the interval $J_0$ to be finite; (Remark \ref{Remark 3.2}). Finally in Section \ref{Section 4}, we consider applications to problems in near field geometric optics, including more precise versions of the flat target cases in \cite{LT2016}.

\vskip10pt
%%%%%%%%%%%%%%%%%%%%%%%%%%%%%%%%%%%%%%%%%%%%%%%%%%%%%%%%%%%%%%%%%%%%%%%%%%%%%%%%%%%%%%%%%%%%%%%%%%%%%%%%%%%%%%%%%%%
\section{Construction of uniformly $g$-convex function}\label{Section 2}
In this section, we shall construct a uniformly $g$-convex function approximately satisfying the second boundary condition \eqref{Second boundary data}, in preparation for the homotopy argument in Section \ref{Section 3}.

\subsection{Initial construction}\label{Section 2.1}

We first recall some convexity notions with respect to the generating function in \cite{T2014} and then construct an initial uniformly elliptic function $u^0$ whose $Tu^0$ mapping over $\Omega$ is a subset of $\Omega^*$.

We recall from \cite{T2014} that a function $u\in C^0(\Omega)$ is $g$-convex, if for each $x_0\in\Omega$,
there exists $y_0 \in \mathbb{R}^n$,  $z_0 \in I(\Omega,y_0) =  \cap_{x\in\Omega} I(x,y_0)$ such that $u(x_0)=g(x_0,y_0,z_0)$ and $u(x) \ge g(x,y_0,z_0)$ for all $x\in\Omega$. If $u(x) > g(x,y_0,z_0)$ for all $x\ne x_0$, then we call u strictly $g$-convex. If a $g$-convex function $u$ is differentiable at $x_0$, then $y_0=Tu(x_0)=Y(x_0,u(x_0),Du(x_0))$, while if $u$ is twice differentiable at $x_0$, then
\begin{equation}\label{locally g-convex}
D^2u(x_0)\ge g_{xx}(x_0,y_0,z_0).
\end{equation}
A function $u\in C^2(\Omega)$ satisfying \eqref{locally g-convex} for all $x_0\in \Omega$ is called locally $g$-convex in $\Omega$. Moreover, the inequality \eqref{locally g-convex} implies that a locally $g$-convex function $u$ of \eqref{GPJE} is automatically degenerate elliptic. We call a $g$-convex function $u\in C^2(\Omega)$ uniformly $g$-convex if the inequality \eqref{locally g-convex} is strict, that is $u$ is also elliptic. Correspondingly, the function $g_0 = g(\cdot, y_0,z_0)$ is called a $g$-affine function, which is a $g$-support of $u$ at $x_0$ if $u(x_0)=g_0(x_0)$ and $u(x)\ge g_0(x)$ for all $x\in \Omega$.

We also recall the corresponding notion of  $g$-convexity for a domain $\Omega$ in \cite{T2014}. A domain $\Omega$ is $g$-convex (uniformly $g$-convex) with respect to $y_0\in  \mathbb{R}^n$, $z_0\in I(\Omega, y_0)$, if the image $Q_0(\Omega):= -g_y/g_z (\cdot, y_0,z_0)(\Omega)$ is convex (uniformly convex) in $\mathbb{R}^n$. From Lemma 2.4 in \cite{T2014}, under the assumptions that A1, A2 and A1* hold in $\mathcal{U}$,  a $C^2$ (connected) domain $\Omega$ being $g$-convex (uniformly $g$-convex) with respect to $y_0, z_0$ is equivalent to
\begin{equation}\label{g-convex domain}
[D_i\gamma_j(x)-g_{ij,p_k}(x,y_0,z_0)\gamma_k(x)]\tau_i\tau_j \ge 0, \ \ (>0),
\end{equation}
for all $x\in \partial\Omega$, unit outer normal $\gamma$ and unit tangent vector $\tau$. Then we see that a $C^2$ domain $\Omega$ is $Y$-convex (uniformly $Y$-convex) with respect to $\Omega^*\times J$ if $\Omega$ is $g$-convex (uniformly $g$-convex) respect to $y_0$ and $z_0$ for all $y_0\in \Omega^*$ and $z_0 = g^*(x,y_0,u_0)$ for all $x\in \Omega$ and $u_0\in J$. Conversely if $\Omega$ is $Y$-convex (uniformly $Y$-convex) with respect to $\Omega^*\times J$, then $\Omega$ is $g$-convex (uniformly $g$-convex) with respect to $y_0$ and $z_0$ for all $y_0\in \Omega^*$ and $z_0$ satisfying $g(\cdot, y_0,z_0)(\Omega)\subset J$. We also recall that for generating functions, the notions of $Y$*-convexity are equivalent to $g^*$-convexity.

Note that the local $g$-convexity of a function $u$ on a $g$-convex domain $\Omega$ can imply its global $g$-convexity. In particular, if we assume $g$ satisfies conditions A1, A2, A1*, A3w and A4w, $\Omega$ is $g$-convex with respect to each point in $(Y,Z)(\cdot, u,Du)(\Omega)$, $u\in C^2(\Omega)$ is locally $g$-convex in $\Omega$, (and $\Gamma$ is sufficiently large), then $u$ is $g$-convex in $\Omega$; see Lemma 2.1 in \cite{T2014}.
We remark also that condition A4w is removed in \cite{T2014-1}, Lemma 2.1, provided  $\Omega$ is $g$-convex with respect to each point $y\in Tu(\Omega)$, $z\in g^*(\cdot,y,u)(\Omega)$ and that the largeness of $\Gamma$ is ensured by assuming $u(\Omega) \subset\subset J(\Omega, Tu(\Omega))$ and $(x,u(x),p) \in \mathcal U$ for all $x\in\Omega$ and $p$ in the convex hull of $Du(\Omega)$. As a consequence, elliptic solutions $u$ of the second boundary value problem \eqref{GPJE}, \eqref{Second boundary data}, satisfying 
\begin{equation}\label{2.3}
[\inf u -K_0d, \sup u +K_0d] \subset J_0,
\end{equation}
\noindent where $d ={\rm diam} \Omega$, will be strictly $g$-convex under the hypotheses of Theorem \ref{Th1.1}. 
More generally if we strengthen the convexity assumption on $\Omega$ in Theorem \ref{Th1.1} so that $\Omega$ is  $g$-convex with respect to all $y\in \Omega^*$ and $z = g^*(x,y,u)$ for all $x \in \Omega$, $u\in J_0$, then elliptic solutions $u$ of  \eqref{GPJE}, \eqref{Second boundary data}, satisfying $u(\Omega) \subset\subset J_0$ will be strictly $g$-convex. 

Let $u\in C^0(\Omega)$ be $g$-convex in $\Omega$. The $g$-normal mapping of $u$ at $x_0\in \Omega$ is the set
$$Tu(x_0)=\{y_0\in  \Gamma_{z,\Omega} |\ u(x)\ge g(x, y_0, g^*(x_0,y_0,u(x_0)))\ {\rm for\ all}\ x\in \Omega \}.$$
For $E\subset \Omega$, we denote $Tu(E)=\cup_{x\in E}Tu(x)$.
When $u$ is differentiable, $Tu$ agrees with the previous terminology that $Tu=Y(x,u,Du)$. In general, we only have
$$Tu(x_0)\subset Y(x_0,u(x_0), \partial u(x_0)),$$
where $\partial u$ denotes the subdifferential of $u$. 
However, if the generating function satisfies the conditions A1, A2, A1*, A3w, (and again $\Gamma$ is sufficiently large), we then have for $g$-convex $u\in C^0(\Omega)$,
$$Tu(x_0)= Y(x_0,u(x_0), \partial u(x_0)),$$
 for any $x_0\in \Omega$; see Lemmas 2.2 in \cite{T2014, T2014-1} for detailed statements.  The reader can refer to \cite{T2014, T2014-1, GK2015} for the more detailed $g$-convexity theory related to generating functions.

Next, we show how to construct a uniformly $g$-convex function $u^0$ by a smooth perturbation of a $g$-affine function $g_0=g(\cdot,y_0,z_0)$.  Such a construction has already been established in \cite{JT-Pogorelov}. One can refer to Lemma 2.1 in \cite{JT-Pogorelov} for more details. We just sketch the construction of $u^0$ for completeness. Set $\Gamma(\Omega,\Omega^*)=\{(x,y,z)\in \Gamma| \ x\in \Omega, y\in \Omega^*, z\in I(x,y)\}$, suppose $g_0=g(\cdot, y_0,z_0)$ is a $g$-affine function on $\bar \Omega$, for $(y_0,z_0)\in \Omega^* \times I(\Omega, y_0)$. Now let $u^0 = g_\rho$ be the $g^*$-transform, introduced in \cite{T2014}, of the function
$$v_\rho(y)=z_0  - \sqrt{\rho^2 - |y-y_0|^2}$$
given by
\begin{equation}\label{2.4}
g_\rho(x)=v_\rho^*(x)=\sup_{y\in B_\rho}g(x,y,v_\rho(y)),
\end{equation}
where $B_\rho=B_{\rho}(y_0)$ and $\rho$ is sufficiently small to ensure that $\Gamma_0=\bar\Omega \times \bar B_\rho\times [z_0-\rho, z_0]\subset \bar \Gamma(\Omega,\Omega^*)$. Then $u^0$ is a uniformly $g$-convex function in $\bar \Omega$, with  image 
\begin{equation}
\omega^*:=Tu^0(\Omega)\subset B_{\rho}(y_0)\subset \Omega^*,
\end{equation}
where $Tu^0 = Y(\cdot,u^0,Du^0)$ is a diffeomorphism between $\Omega$ and $\omega^*$. We can also estimate
\begin{equation}
g_0- \sup_{\Gamma_0}|g_y| \rho \le g_\rho\le g_0,
\end{equation}
which shows that $g_\rho$ converges uniformly to $g_0$ as $\rho$ tends to zero. 

 We remark that under the hypotheses of Theorem \ref{Th1.1}, we can also determine a suitable $u^0$ so that $Tu^0(\Omega) = B_{\rho}(y_0)$ for $\rho$ sufficiently small. This is accomplished by using domain foliation in Section \ref{Section 3},  similarly to \cite{TW2009} and \cite{LT2016}.

\subsection{A fundamental geometric characterization}\label{Section 2.2}

In this subsection, we derive a geometric property of the uniformly $Y$-convex domain $\Omega$, which will be used in Section \ref{Section 2.3} to extend the initial construction $u^0$ in Section \ref{Section 2.1} from $\Omega$ to a neighbourhood $\Omega^\delta=\{x\in \mathbb{R}^n|\ {\rm dist}(x,\Omega)<\delta\}$ for some $\delta>0$. 

Suppose that the domains $\Omega$, $\Omega^*$ and generating function $g$  satisfy conditions A1, A2, A1*, A3w, A5, and $\Omega$, $\Omega^*$ are respectively uniformly $Y$-convex, $Y^*$-convex with respect to each other and any interval $J \subset\subset J_0$. We denote the unit outer normal of $\partial \Omega$ by $\gamma$ and let $u \in C^2(\bar\Omega)$ be
 $g$-convex and $g_0 = g(\cdot,y_0,z_0)$ 
 be a $g$-affine function defined on $\bar\Omega$ such that
 \begin{equation}\label{2.6} 
  Tu(\bar\Omega)\cup\{y_0\}\subset \Omega^*, \  u(\bar\Omega), g_0(\bar\Omega) \subset J_0. 
\end{equation}
 Letting $h = u-g_0$ denote the height of $u$ above $g_0$, we also assume for some boundary point $x_0\in\partial \Omega$
 \begin{equation}\label{h}
 h(x_0) = 0, \ Dh(x_0) = -s\gamma_0,
 \end{equation}
 where $\gamma_0 = \gamma(x_0)$ and $s$ is a positive constant.
 
 Note that in order to ensure the inclusions,  $u, g_0(\bar\Omega) \subset J_0$, we may assume $[u_0 - K_0 d, u_0 + K_0 d] \subset J_0$, where $u_0 = u(x_0) = g_0(x_0)$.
  
  The following key lemma shows that $h$ is positive away from $x_0$.

\begin{lemma}\label{key lemma}

Under the above hypotheses,  the functions $g_0,u$ satisfy
\begin{equation}\label{h>0}
g_0(x) < u(x), \quad {\rm for\ all} \  x\in\bar \Omega\backslash\{x_0\}.
\end{equation}
\end{lemma}

Lemma \ref{key lemma} is a consequence of the special case when $u = g_1 =g(\cdot, y_1,z_1)$ is also $g$-affine. The property \eqref{h>0} in Lemma \ref{key lemma} then asserts that the domain $\Omega$ lies strictly on one side of the level set of the function $h=g_1-g_0$, passing through $x_0$, which is tangential to $\partial \Omega$ at $x_0$ by virtue of  \eqref{h}. This may  be  proved by modification of the proof of the corresponding inequality in the optimal transportation case, namely inequality (7.3) in \cite{TW2009}, which originated in \cite{TW2009-1}. The proof  presented here follows the approach in \cite{TW2008, T2014, T2014-1}, using a fundamental differential inequality for the function $h$. (Note that \cite{TW2008} should be substituted for reference [21] in \cite{TW2009}, and $c$ should be replaced by $-c$ in inequality (7.3) in \cite{TW2009} and its subsequent proof). 

\begin{proof}[Proof of Lemma \ref{key lemma}.]
First we suppose that $u$ is uniformly $g$-convex and $g_0 = g(\cdot,y_0,z_0)$ satisfies \eqref{2.6} but not necessarily \eqref{h}
 and let $x$ be some point in $\bar \Omega\backslash\{x_0\}$ so that by the uniform $g$-convexity of $\Omega$,  the open $g$-segment, with respect to $y_0,z_0$, joining $x$ to $x_0$ lies in $\Omega$. Setting $q_0= Q(x_0,y_0,z_0)$, $q= Q(x,y_0,z_0)$, $q_ t= (1-t)q_0 + tq$, for $0\le t\le 1$ we thus have  
$$x_t: = X(q_t, y_0, z_0) \in \Omega.$$
Defining the function $h_0$ on $[0,1]$ by $h_0(t) = h(x_t)$, we then have from \cite{T2014}, the differential inequality
\begin{equation}\label{h'' inequality}
h^{\prime\prime}_0 >0,
\end{equation}
whenever $h_0=h^\prime_0 =0$. 

Now let us suppose $u=g_1= g(\cdot, y_1,z_1)$ is $g$-affine with $h_0(0) = 0$, $h^\prime_0(0)> 0$ and let $u =g_\rho$ denote the uniformly $g$-convex approximation \eqref{2.3} to $g_1$, (with $y_0, z_0$ replaced by $y_1,z_1$) and set $h_\rho = g_\rho -g_0$,  $h_{\rho,0} =(h_\rho)_0$. Following \cite{T2014-1}, we suppose $h(x) = h_0(1) \le 0$. Then since we must have $h_0 > 0$, for small $t$, it follows that there exists
$\delta > 0$ so that $z_\delta = z_0-\delta \in I(\Omega,y_0)$, 
$\Omega$ is uniformly $g$-convex with respect to $y_0, z_\delta$ 
and, when $z_0$ is replaced by $z_\delta$ in $g_0$,
the function $h_{\rho,0}$ takes a zero maximum at some $t^*\in (0,1)$, for sufficiently small $\rho$, with $h_{\rho,0}(0), h_{\rho,0}(1) < 0$, which contradicts \eqref{h'' inequality}.  Consequently, using the formula $h^\prime_0(0) = D_\eta h(x_0)$, where the vector $\eta$ is given by
\begin{equation}
\eta_j = -g_z E^{i,j}(x_0,y_0,z_0)[q_i-(q_0)_i],
\end{equation}
where   $(E^{i,j})=(D_{p_j}Y^i)=E^{-1}$, we obtain $h(x) > 0$ for 
$x \in \bar\Omega\backslash\{x_0\}$, provided $h(x_0)$ = 0 and $D_\eta h(x_0) > 0$. 
Now from \eqref {h}, we have 
\begin{equation}\label{hard to understand}
D_\eta h(x_0) = g_z E^{i,j}(x_0,y_0,z_0)[q_i-(q_0)_i](\gamma_0)_j >0,
\end{equation}
and we conclude \eqref{h>0} since the vector $E^{-1}(x_0,y_0,z_0)\gamma_0$ is a positive multiple of the outer normal at $q_0$ to the uniformly convex domain $Q(\cdot, y_0,z_0)$. By replacing $g$-convex $u$ by its $g$-support at $x_0$, we obtain Lemma \ref{key lemma} in its full generality.
\end{proof}

\begin{remark}
When $u$ is uniformly $g$-convex or A3 or A4w hold, we do not need to use the approximation \eqref{2.3} in the above proof. This is automatic  when $u$ is uniformly $g$-convex while if A3 holds we also have the strict inequality \eqref{h'' inequality} when $u$ is only assumed $g$-convex. In the case A4w, we have from \cite{T2014} the differential inequality 
$$ h^{\prime\prime}_0 \ge -K|h^\prime_0|,$$
whenever $h_0\ge 0$, for some positive constant $K$, and we infer $h(x) > 0$ directly, without adjusting $z_0$, as in \cite{T2014}. Moreover if the strict version A4 of condition A4w holds, then we have again the strict inequality
 \eqref{h'' inequality} for $g$-convex $u$.
\end{remark}

\subsection{Extension}\label{Section 2.3}

In this subsection, we use the property \eqref{h>0} to extend our initial construction $u^0$ from $\Omega$ to $\Omega^\delta=\{x\in \mathbb{R}^n|\ {\rm dist}(x,\Omega)<\delta\}$, following the argument in the optimal transportation case \cite{TW2009}.
Note that we only need to extend $\Omega$ to a sufficiently small neighbourhood, so $\delta$ can be chosen sufficiently small. We may assume $u^0\in C^{\infty}(\bar \Omega)$ by approximation and make the extension using envelopes of $g$-affine functions. Recall that a $g$-affine function in $\Omega$ has the form $g_0 = g(\cdot, y_0,z_0)$ for $y_0\in\mathbb{R}^n$, $z_0\in I(\Omega, y_0)$ and its $g$-normal mapping image $Tg_0 (\Omega) = \{y_0\}$. We consider the following admissible set
$$\mathcal{S}=\{g_0(x)|\  g_0(x)\ {\rm is}\   g {\rm -affine \ in \ \Omega^\delta,} \ g_0\le u^0\ {\rm in}\ \Omega, \ Tg_0(\Omega)\subset \Omega^*\}, $$
and take
\begin{equation}
u_1(x)=\sup_{g_0\in \mathcal{S}}\{u^0, g_0\}, \quad x\in \Omega^\delta.
\end{equation}
Then the following lemma, extending Lemma 7.1 in \cite{TW2009}, describes the properties of  the function $u_1$.

\begin{lemma}\label{Lemma extension}
Assume that the domains $\Omega$, $\Omega^*$ and generating function $g$  satisfy conditions A1, A2, A1*, A3w, A5, and
 $\Omega$, $\Omega^*$ are respectively uniformly $Y$-convex, uniformly $Y^*$-convex with respect to each other and any interval $J \subset\subset J_0$. Then, for sufficiently small $\delta$ and $u^0$  satisfying \eqref{2.3}, the function $u_1$ is a 
 $g$-convex extension of $u^0$ from $\Omega$ to $\Omega^\delta$, whose $g$-normal image under $u_1$ is $\bar \Omega^*$. Moreover, for any $x\in \Omega^\delta - \bar \Omega$, there exist unique points $x_b \in \partial \Omega$, $y_b = Tu_1(x), \in \partial\Omega^*$, such that $Tu_1(\ell_{y_b}) =y_b$, where $\ell_{y_b}$ is the open $g$-segment with respect to $y_b$, $z_0 = g^*(x_b,y_b,u^0(x_b))$, joining $x_b$ to $x$, with the resultant mappings being $C^2$ diffeomorphisms from $\partial\Omega^r$ to $\partial\Omega$, $\partial\Omega^*$ respectively, for any $r<\delta$.
\end{lemma}

\begin{proof}
We take any $g$-affine function $g_0=g(x,y,z_0)$ in $\mathcal{S}$, with $y\in \Omega^*\backslash\omega^*$. Since $u^0$ is uniformly $g$-convex, by decreasing $z_0$, thereby increasing $g_0$, the graph of $g_0$ will touch $u^0$ from below at a point $x_b\in \partial\Omega$. Accordingly we may assume that the $g$-affine function $\bar g\in \mathcal{S}$, given by
$$\bar g(x):=\bar g_{x_b,y,z}(x)=g(x,y,z),$$
for the same $y$ in $g_0$ and some $z<z_0$, touches $u^0$ from below at $x_b\in \partial\Omega$, whence 
$z=z_y=g^*(x_b, y_b,u^0_b)$, where $u^0_b = u^0(x_b)$.
Since $\bar g \le u^0$ in $\Omega$, $\bar g(x_b)=u^0_b$, the point $y$ must lie on $\ell_{x_b}^*$,  which is the image under $Y(x_b,u^0_b,\cdot)$ of the straight line from $Du^0_b =Du^0(x_b)$ with the slope $\gamma_0$, that is
$$y=Y(x_b, u^0_b, Du^0_b+s\gamma_0)\in \ell_{x_b}^*,$$
for some $s\ge 0$, $\gamma_0 = \gamma(x_b)$. Moreover, $\ell_{x_b}^*$ starts at the point $y_{0,b}=Tu^0(x_b)$.
 Conversely, for any $x_b\in \partial \Omega$, $y\in \ell_{x_b}^*$, we have from \eqref{h>0},
 \begin{equation}
 \bar g(x) = g(x,y,z_y) <  u^0(x), \quad {\rm for}\ x\in \bar \Omega\backslash \{x_b\}.
\end{equation}
This proves that $u_1$ is indeed a $g$-convex extension of $u^0$ from $\Omega$ to $\Omega^\delta$.

To proceed further, from the uniform $Y^*$-convexity of $\Omega^*$,  $\ell_{x_b}^*$ intersects with $\partial \Omega^*$ at the unique point $y_b$, and from the uniform $g$-convexity of $u^0$, $\ell_{x_b}^*$ only intersects with $\partial \omega^*$ at the initial point $y_{0,b}$.
We then restrict $\ell_{x_b}^*$ to the segment joining $y_{0,b}$ and $y_b$. From the argument above, the mapping from $x_b$ to $y_b$ is onto $\partial \Omega^*$. From \eqref{h>0}, it is also one-to-one as the $g$-affine function $\bar g$ cannot meet $\partial\Omega$ at another point $x'$. It follows then the mapping from $x_b$ to $y_b$ is a $C^2$ diffeomorphism from $\partial\Omega$ to $\partial \Omega^*$.
Next, if $B_r$ is a sufficiently small exterior tangent ball of $\Omega$ at $x_b$, it will also be uniformly $Y$-convex. 
Defining $z_b= g^*(x_b,y_b,u^0_b)$, we then have from \eqref{h>0} again that
\begin{equation}
g(x,y_b,z_b) \ge  g(x,y,z_y),
\end{equation}
for all $x\in B_r$, $y\in \ell^*_{x_b}$. Note that $\gamma _0$ is now the inner normal at $x_b$ to $B_r$.
Thus, we have
\begin{equation}\label{u1new}
u_1 = \max_{x_b \in \partial\Omega}\{u^0, \bar g_{x_b,y_b,z_b}\}.
\end{equation}

To complete the proof of Lemma \ref{Lemma extension}, we need to show that for each $x\in \Omega^\delta\backslash\Omega$, there exists a unique $x_b\in \partial\Omega$, where the maximum in \eqref{u1new} is attained. For this, we invoke the $g$-transform of $u_1$,
$$v^0(y)=\sup_{x\in \Omega^\delta}\{g^*(x,y,u_1(x))\},\quad y\in \Omega^*,$$
which extends the $g$-transform of $u^0$ in $\omega^*$. Moreover, by $\bar g \le u^0$ in $\Omega$, we see that for $y\in \ell^*_{x_b}$, the supremum is attained at $x_b$. Hence, we have
\begin{equation}\label{v0y}
v^0(y)=g^*(x_b, y, u^0_b), \quad {\rm for\  all}\ y\in  \ell^*_{x_b}.
\end{equation}
One easily verifies that $v^0$ is smooth in $\bar \Omega^* \backslash\partial\omega^*$. Using \eqref{v0y} and arguing as before, we infer that for any point $x\in \Omega^\delta - \bar\Omega$, there exists a unique point $y_b\in \partial\Omega$ such that
\begin{equation}\label{g*y}
g^*_{x,y_b,u_0}(y) \le v^0(y),\quad \forall y\in \bar \Omega^*,
\end{equation}
for some $u_0$. Moreover, $x$ lies on $\ell_{y_b}$ which is the image under $X(y_b, v^0(y_b), \cdot)$ of the straight line segment from $Dv^0(y_b)$ with the slope $\gamma^*(b_b)$, namely,
$$x=X(y_b, v^0(y_b), Dv^0(y_b)+s\gamma^*(y_b))\in \ell_{y_b},$$
where $\gamma^*$ denotes the unit outer normal to $\partial \Omega^*$, $s\in [0,\bar \delta]$, and $\bar \delta$ is a small constant. Note that $x_b=X(y_b, v^0(y_b), Dv^0(y_b))$. From \eqref{g*y}, we see that the maximum in \eqref{u1new} is attained at $x_b$, $y_b$, so
$$u_1(x) = \bar g_{x, y_b, z_b}(x),\quad x\in \ell_{y_b},$$
with $Tu_1(\ell_{y_b}-\{x_b\})=y_b$, $Tu_1(x_b)=\ell^*_{x_b}$. From the obliqueness of $\ell_{y_b}$ on $\partial \Omega$, we have that the mapping from $x\in \Omega^r$ to $x_b$ is one-to-one for sufficiently small $r$. This completes the proof of Lemma \ref{Lemma extension}.
\end{proof}

\subsection{Adjustment and mollification}\label{Section 2.4}
In this subsection, we will make further adjustment and mollification of the extended function $u_1$. From the construction of $u_1$ in the previous subsection, we know that $u_1$ is smooth in $\Omega^\delta\backslash\partial\Omega$. Modifying $u_1$ in $\Omega^\delta\backslash \Omega$, by defining
\begin{equation}
u=\left\{
\begin{array}{ll}
u, & x\in \Omega,\\
u_1+td^2, & x\in \Omega^\delta\backslash \Omega,
\end{array}
\right.
\end{equation}
where $t$ is a small positive constant and $d$ denotes the distance from $\Omega$. It is readily seen that, for $\delta$ sufficiently small,
\begin{equation}
[D_{ij}u - A_{ij}(\cdot, u, Du)]\xi_i\xi_j \ge \lambda_0,
\end{equation}
in $\Omega^\delta\backslash \partial\Omega$ for some positive constant $\lambda_0$ and any unit vector $\xi$. Then the image of the $g$-normal mapping of $u$ in $\Omega^{\delta}$ is a small perturbation of $\Omega^*$ containing $\Omega^*$.

We can now mollify the function $u$ by
\begin{equation}\label{mollification}
u_\epsilon(x)=\rho \ast u = \int_{\mathbb{R}^n} \epsilon^{-n} \rho(\frac{x-y}{\epsilon})u(y)dy=\int_{\mathbb{R}^n}\rho(y)u(x-\epsilon y)dy,
\end{equation}
where $\rho\in C^\infty_0(B_1(0))$ is a nonnegative symmetric mollifier satisfying $\int_{B_1(0)}\rho =1$, $\epsilon$ is a positive constant. Taking $\epsilon<\delta/2$ sufficiently small and $x\in \Omega^{\frac{\delta}{2}}$, we will show that $u_\epsilon(x)$ is uniformly $g$-convex in $\Omega^{\frac{\delta}{2}}$. Note that the image of the $g$-normal mapping of $u_\epsilon$ in $\Omega^{\frac{\delta}{2}}$ is a smooth perturbation of $\Omega^*$. First, we recall some properties of $u_\epsilon$ from \cite{TW2009},
\begin{equation}\label{property gradient}
Du_\epsilon (x) = \int_{\mathbb{R}^n}\rho(y) Du(x-\epsilon y)dy,
\end{equation}
\begin{equation}\label{property second derivative}
D^2u_\epsilon (x) \ge \int_{\mathbb{R}^n\backslash \mathcal{P}_{x,\epsilon}}\rho(y) D^2u(x-\epsilon y)dy,
\end{equation}
where $\mathcal{P}_{x,\epsilon}:=\{y\in \mathbb{R}^n | \ x-\epsilon y \in \partial\Omega\}$.
We then divide $\Omega^{\frac{\delta}{2}}$ by $\Omega^{\frac{\delta}{2}}=U_1\cup U_2 \cup U_3$ where $U_1:=\{x\in \Omega^{\frac{\delta}{2}}|\ {\rm dist}(x,\partial\Omega) \ge \epsilon\}$, $U_2:=\{x\in \Omega^{\frac{\delta}{2}}|\ {\rm dist}(x,\partial\Omega) \in (\epsilon',\epsilon)\}$ and $U_3:=\{x\in \Omega^{\frac{\delta}{2}}|\ {\rm dist}(x,\partial\Omega) \le \epsilon'\}$, with $\epsilon'=(1-\sigma)\epsilon$ and $\sigma \in (1/2, 1)$ is a constant close to $1$. It is clear that $u_\epsilon$ is smooth and uniformly $g$-convex in $U_1$ provided $\epsilon$ is sufficiently small. Also, by choosing $\sigma$ sufficiently close to $1$, for any $x\in U_2$, from \eqref{mollification} and \eqref{property gradient}, $u_\epsilon$, $Du_\epsilon$ are small perturbations of $u$ and $Du$, respectively. By \eqref{property second derivative}, we have $u_\epsilon$ is smooth and uniformly $g$-convex in $U_2$. We next check the uniform $g$-convexity of $u_\epsilon$ in $U_3$. For any point $x_0\in U_3$, without loss of generality, we choose the nearest point of $x_0$ on $\partial\Omega$ to be the origin, and choose the direction pointing from $0$ to $x_0$ to be $e_n$ so that $\partial\Omega$ is tangent to $\{x_n=0\}$ and $x_0 = (0,\cdots, 0, x_{0,n})$. We choose a unit vector $\tau$ tangential to $\partial\Omega$ at $0$. Without loss of generality, we can assume $\tau =(1,0,\cdots, 0)$. Then we need to prove
\begin{equation}
D_{11}u_\epsilon (x_0) - A_{11}(x_0, u_\epsilon (x_0), Du_\epsilon (x_0)) >0.
\end{equation}
By the choice of coordinates, $D_1u_\epsilon(x_0)$ is a small perturbation of $D_1u(x_0)$. Notice that, from \eqref{mollification}, $u_\epsilon(x_0)$ is also a small perturbation of $u(x_0)$. It now suffices to prove
\begin{equation}\label{tangential uniform g-convexity}
D_{11}u_\epsilon (x_0) - A_{11}(x_0, u (x_0), D_1u(x_0), D'u_\epsilon (x_0)) >0.
\end{equation}
where $D'u_{\epsilon}=(D_2u_\epsilon, \cdots, D_nu_\epsilon)$. From A3w, $A_{11}$ is convex with respect to $D'u_\epsilon$. Therefore, we have
\begin{equation}\label{A integral inequality}
\begin{array}{ll}
\!\!&\!\!\displaystyle A_{11}(x_0, u (x_0), D_1u(x_0), D'u_\epsilon (x_0))\\
\le \!\!&\!\!\displaystyle \int_{\mathbb{R}^n} \epsilon^{-n} \rho(\frac{x_0-y}{\epsilon})A_{11}(x_0, u (x_0), D_1u(x_0), D'u(y))dy.
\end{array}
\end{equation}
From \eqref{A integral inequality}, and again from \eqref{property second derivative}, we have \eqref{tangential uniform g-convexity} holds. From the property of the second integral in (7.19) in \cite{TW2009}, we also have
\begin{equation}\label{normal uniform g-convexity}
D_{nn}u_\epsilon (x_0) - A_{nn} (x_0, u_\epsilon (x_0), Du_\epsilon (x_0))\ge K,
\end{equation}
for sufficiently large $K$, provided $\epsilon$ is sufficiently small. Combining \eqref{tangential uniform g-convexity} and \eqref{normal uniform g-convexity}, we know that $u_\epsilon$ is uniformly $g$-convex in $U_3$. Thus, we have proved that the function $u_\epsilon$ is smooth and uniformly $g$-convex in $\Omega^{\frac{\delta}{2}}$ for $\epsilon <\frac{\delta}{2}$ sufficiently small.

Then by appropriate adjustment of the domain $\Omega$, we have the following lemma, which gives the existence of  uniformly $g$-convex smooth functions with approximating target domains.

\begin{lemma}\label{Lemma - initial solution for approximating domains}
Let the domains $\Omega$, $\Omega^*$ and the generating function $g$ satisfy the hypotheses of Theorem \ref{Th1.1}. Then for any $\epsilon>0$, there exists a uniformly $g^*$-convex $C^4$ approximating domain $(\Omega^{*})^\epsilon$ lying within the distance $\epsilon$ of $\Omega^*$, together with a uniformly $g$-convex function $u\in C^4(\bar \Omega)$ satisfying the boundary condition \eqref{Second boundary data} for $(\Omega^{*})^\epsilon$.
\end{lemma}
Note that if we do not make an adjustment of the domain $\Omega$, we get a uniformly $g$-convex function for  approximating domains for both $\Omega$ and $\Omega^*$.

%%%%%%%%%%%%%%%%%%%%%%%%%%%%%%%%%%%%%%%%%%%%%%%%%%%%%%%%%%%%%%%%%%%%%%%%%%%%%%%%%%%%%%%%%%%%%%%%%%%%%%%%%%%%%%%%%%%
\section{Proof of existence theorems}\label{Section 3}

In this section, we give the proof of the existence result, Theorem \ref{Th1.1}, utilizing the method of continuity, supplemented by  degree theory for nonlinear oblique boundary value problems, as in \cite{FP1986, L1989, FP1993, LT2016, LLN2015}. From Lemma \ref{Lemma - initial solution for approximating domains}, we can assume initially that there exists a uniformly $g$-convex function $u_0\in C^4(\bar\Omega)$ satisfying \eqref{Second boundary data}, that is $Tu_0(\Omega) = \Omega^*$. From our construction 
in Section \ref{Section 2}, we can also assume the inclusion \eqref{2.3}.
We will also need that  the second boundary value condition \eqref{Second boundary data} implies a nonlinear oblique boundary condition for uniformly elliptic functions $u$, \cite{Tru2008}. In particular, letting   $\phi^*\in C^2(\mathbb{R}^n)$
 be a defining function for $\Omega^*$, satisfying $\phi^*=0$, $D\phi^*\neq 0$ on $\partial\Omega^*$, $\phi^*<0$ in 
$\Omega^*$, $\phi^*> 0$ in $\mathbb{R}^n - \bar\Omega^*$ and setting, 
\begin{equation}\label{G}
G(x,u,p)=\phi^* \circ Y(x, u, p),
\end{equation}
for $(x,u,p) \in\mathcal U$, we obtain  
\begin{equation}\label{BC}
G[u]:=G(\cdot,u, Du)=0, \quad {\rm on} \ \partial\Omega,
\end{equation}
together with the obliqueness condition,
\begin{equation}
G_{p}(\cdot, u, Du)\cdot \gamma > 0, \quad {\rm on} \ \partial\Omega.
\end{equation}
Furthermore, the $Y^*$-convexity (uniform  $Y^*$-convexity) of $\Omega^*$ with respect to $\Omega\times J$ implies that 
$G$ is convex (uniformly convex) in $p$ for $x\in \partial\Omega$, $u \in J$, $Y(x,u,p)\in \partial \Omega^*$, and additionally
 the uniform $Y^*$-convexity of $\Omega^*$ implies that $G=\phi^* \circ Y$ is uniformly convex in $p$ when $Y$ lies in some neighbourhood $\mathcal{N}^* = \{|\phi^*|< \delta\}$ of $\partial\Omega^*$, for some $\delta > 0$. In particular, these properties are essential for showing that the initial problem in our homotopy family \eqref{Homotopy family} is uniquely solvable; (see Lemma \ref{Uniqueness lemma} below). We remark also that the boundary condition \eqref{Second boundary data} is implied by  \eqref{BC} when $Tu$ is also one-to-one on $\bar\Omega$ and this would follow from the ellipticity of $u$ on $\bar\Omega$, together with the $g$-convexity of 
 $\Omega$ with respect to some $y_0,z_0$. 

 We now consider for $0\le t \le 1$, $\tau >0$ and $\epsilon>0$, the family of generated prescribed Jacobian equations,
\begin{equation}\label{Homotopy family}
 |\det DTu| = e^{[\tau (1-t)+\epsilon](u-u_0)}[tf+(1-t)f^*\circ Tu_0|\det(DTu_0)|]/f^*\circ Tu, \quad {\rm in} \ \Omega,\\
 \end{equation}
for elliptic solutions $u \in C^2(\bar\Omega)$, with one-jet $J_1[u](\Omega)\subset\subset \mathcal U$, and range
$u(\Omega) \subset\subset J_0$,  where $Tu = Y(\cdot, u, Du)$.  Here we call the solution $u$ elliptic if $D^2u > A(\cdot,u,Du)$, so that equation \eqref{Homotopy family} can still be written in the forms \eqref{PJE} and \eqref{GPJE}; 
(see equation \eqref{equation MA form}.

\begin{lemma}\label{Uniqueness lemma} 
Under the hypotheses of Theorem \ref{Th1.1},  for sufficiently large $\tau$, $u_0$ is the unique elliptic solution of the second boundary value problem \eqref{Homotopy family}, \eqref{Second boundary data} at $t=0$. 
\end{lemma}

\begin{proof}
First suppose $u\in C^2(\bar\Omega)$ is an elliptic solution of \eqref{Homotopy family}, \eqref{Second boundary data} for some $t\in [0,1]$, so that $Tu$ is a diffeomorphism  from $\Omega$ to $\Omega^*$. Multiplying by $f^*\circ Tu$ and integrating \eqref{Homotopy family} over $\Omega$, we obtain, from the change of variables formula and the conservation of energy \eqref{conservation of energy}, 
\begin{equation}\label{using conservation Homo}
 t  \int_{\Omega} \{e^{[\tau (1-t)+\epsilon] (u -u_0)}-1\} f +(1-t)  \int_{\Omega^*}\{e^{[\tau (1-t)+\epsilon] (u -u_0)\circ(Tu_0)^{-1}}-1\}f^* = 0,
 \end{equation}
which implies that $u = u_0$ at some point in $\Omega$. From A1, A5, and the assumption that $u$ lies in the  open interval $J_0$, we have
\begin{equation}
  \sup_{\Omega}|Du| \le K_0,
\end{equation}
where $K_0$ is the constant in A5. From \eqref{2.3} we then have $u(\bar\Omega) \subset J$ for some fixed $J \subset\subset J_0$, so that in particular 
\begin{equation}
  \sup_{\Omega}|u| \le M_0,
\end{equation}
for some fixed constant $M_0$, depending on $u_0, K_0d$ and $J_0$. Thus any elliptic solution $u$ of \eqref{Homotopy family}, \eqref{Second boundary data} satisfies a uniform $C^1$ bound
\begin{equation} \label{gradient bound}
|u|_{1;\Omega} \le C.
\end{equation}
with its one-jet $J_1[u](\bar \Omega)$ lying in a fixed set  $\mathcal U_0 \subset\subset \mathcal U$. Writing equation \eqref{Homotopy family} in the Monge-Amp\`ere form,
\begin{equation}\label{equation MA form}
F[u] := \log \det [D^2u -A(\cdot,u,Du)] =  [\tau (1-t)+\epsilon] (u -u_0) + \log B_t(\cdot,u,Du),
\end{equation}
where
$$B_t = |\det E| [tf+(1-t)f^*\circ Tu_0|\det DTu_0|]/f^*\circ Y,$$
we can then infer higher order estimates for elliptic solutions $u$, which we will need for our continuity argument. 

Returning to the uniqueness assertion in Lemma \ref{Uniqueness lemma} in the case $t=0$, we consider the function $w$ given by
\begin{equation}
w = e^{-\kappa \phi} (u-u_0),
\end{equation}
where $\kappa$ is a positive constant to be fixed later and $\phi \in C^2 (\bar\Omega)$ is a defining function for $\Omega$, satisfying  $\phi=0$ on $\partial\Omega$, $D\phi = \gamma $ on $\partial\Omega$ and $\phi<0$ in $\Omega$.  Assuming that the function $w$ attains its positive maximum at $x_0\in \Omega$, then we have
\begin{equation}\label{formulae at max}
Dw(x_0)=0, \quad {\rm and} \ \ Lw(x_0) \le 0,
\end{equation}
where $L$ is a linearized operator defined by
\begin{equation}
L: = F^{ij}(M[u_0])D_{ij},
\end{equation}
with $F^{ij}(M[u_0])=\frac{\partial F(M[u_0])}{\partial r_{ij}}$, $\{r_{ij}\}=M[u_0]=D^2u_0-A(x,u_0,Du_0)$. By a direct calculation, we have 
\begin{equation}\label{Lwx0}
\begin{array}{ll}
L w(x_0) = \!\!&\!\!  F^{ij}(M[u_0](x_0)) \{e^{-\kappa \phi(x_0)}[(D_{ij}u(x_0)-A_{ij}(x_0,u(x_0),Du(x_0))) \\
         \!\!&\!\!  -(D_{ij}u_0(x_0)-A_{ij}(x_0,u_0(x_0),Du_0(x_0)))] \\
         \!\!&\!\!  + e^{-\kappa \phi(x_0)}[A_{ij}(x_0,u(x_0),Du(x_0))-A_{ij}(x_0,u_0(x_0),Du_0(x_0))\\
         \!\!&\!\!  -2\kappa  D_i\phi(x_0)D_j(u-u_0)(x_0) ] + \kappa [ - D_{ij}\phi(x_0) + \kappa D_i\phi(x_0) D_j\phi(x_0)] w(x_0)\}.
\end{array}
\end{equation}
Using the concavity of ``$\log\det$'' and equation \eqref{equation MA form} at $t=0$, we have
\begin{equation}\label{concavity of F}
\begin{array}{ll}
         \!\!&\!\!  e^{-\kappa \phi(x_0)}F^{ij}(M[u_0](x_0))[(D_{ij}u(x_0)-A_{ij}(x_0,u(x_0),Du(x_0))) \\
         \!\!&\!\!  -(D_{ij}u_0(x_0)-A_{ij}(x_0,u_0(x_0),Du_0(x_0)))] \\
   \ge      \!\!&\!\! e^{-\kappa \phi(x_0)} (F[u(x_0)]-F[u_0(x_0)]) \\
   =         \!\!&\!\! (\tau + \epsilon) w(x_0) + e^{-\kappa \phi(x_0)} [\log B_0(x_0, u(x_0), Du(x_0)) - \log B_0(x_0, u_0(x_0), Du_0(x_0))].
\end{array}
\end{equation}
By the mean value theorem, we have
\begin{equation}\label{mean value theorem}
\begin{array}{ll}
         \!\!&\!\!  A_{ij}(x_0,u(x_0),Du(x_0))-A_{ij}(x_0,u_0(x_0),Du_0(x_0)) \\
   =    \!\!&\!\! D_uA_{ij}(x_0,\hat u(x_0), Du(x_0))(u-u_0)(x_0) + D_{p_k}A_{ij}(x_0,u_0(x_0),\hat p(x_0))D_k(u-u_0)(x_0),
\end{array}
\end{equation}
for all $i, j=1,\cdots, n$, where $\hat u=(1-\theta_1) u+ \theta_1 u_0$, $\hat p = (1-\theta_2)Du+\theta_2 Du_0$, for some $0<\theta_1, \theta_2 <1$. 
Similarly, again by the mean value theorem, we have
\begin{equation}\label{again mean value theorem}
\begin{array}{ll}
         \!\!&\!\!  \log B_0(x_0,u(x_0),Du(x_0))-\log B_0 (x_0,u_0(x_0),Du_0(x_0)) \\
   =    \!\!&\!\! D_u(\log B_0)(x_0,\tilde u(x_0), Du(x_0))(u-u_0) (x_0) \\ 
         \!\!&\!\! + D_{p_k}(\log B_0)(x_0,u_0(x_0),\tilde p(x_0))D_k(u-u_0)(x_0),
\end{array}
\end{equation}
where $\tilde u=(1- \zeta_1) u+ \zeta_1 u_0$, $\tilde p = (1-\zeta_2)Du+\zeta_2 Du_0$, for some $0<\zeta_1, \zeta_2 <1$.
From the equality in \eqref{formulae at max}, we have
\begin{equation}\label{critical point}
D_i(u-u_0)(x_0)=\kappa (u-u_0)(x_0) D_i \phi(x_0), \quad {\rm for} \ i=1,\cdots, n.
\end{equation}
Substituting \eqref{concavity of F}, \eqref{mean value theorem}, \eqref{again mean value theorem} into \eqref{Lwx0} and using \eqref{critical point}, we get
\begin{equation}
L w(x_0) \ge (\tau + \epsilon - C) w(x_0),
\end{equation}
where the constant $C$ depends on $F^{ij}(M[u_0])$, $D_uA$, $D_p A$, $D_uB_0$, $D_pB_0$, $B_0$, $M_0$, $K_0$, $\phi$ and $\kappa$.
By choosing $\tau$ sufficiently large such that $\tau + \epsilon >C$, we get $L w(x_0)>0$. It follows that the function $w$ cannot take a positive maximum in 
$\Omega$. 

Accordingly we suppose that $w$ takes a positive maximum at a  point $x_0\in\partial\Omega$, whence 
\begin{equation}
\beta_0\cdot Dw(x_0)  = e^{-\kappa \phi(x_0)}[\beta_0\cdot D(u-u_0)(x_0) - \kappa \beta_0\cdot\gamma(x_0)(u-u_0)(x_0)] \ge 0, 
\end{equation}
where $\beta_0= G_p(J_1[u_0](x_0))$ and from the obliqueness of $G$ with respect to $u_0$, we have 
$\beta_0\cdot\gamma(x_0) >0$. Now we extend $G$ so that  $G \in C^1(\partial\Omega\times J\times \mathbb{R}^n)$ is convex in $p$ and agrees with \eqref{G} for $|\phi^*\circ Y| < \delta/2$. From the convexity of $G$ and the boundary condition \eqref{BC}, we then have
\begin{equation}
G(x_0, u_0(x_0), Du(x_0)) \ge \beta_0\cdot D(u-u_0)(x_0) \ge \kappa \beta_0\cdot \gamma(x_0)(u-u_0)(x_0).
\end{equation} 
Now for sufficiently large $\kappa$, depending on $G_u, M_0 , K_0$ and $\beta_0\cdot\gamma(x_0)$, we have from \eqref{BC}   again,
\begin{equation}
G(x_0, u_0(x_0), Du(x_0))<  \kappa \beta_0\cdot\gamma(x_0)(u-u_0)(x_0).
\end{equation}
Consequently we must have $u\le u_0$ in $\Omega$ and we immediately conclude $u=u_0$ in $\Omega$ from \eqref{using conservation Homo}.
\end{proof}

Now we can complete the proof of Theorem \ref{Th1.1}. With $\tau$ fixed, in accordance with Lemma \ref{Uniqueness lemma} and $\epsilon$ sufficiently small, say $\epsilon < 1$, we first note from \eqref{gradient bound} and 
$J_1[u](\Omega) \subset \mathcal U_0$ that $|\det E|$ and $B_t$ will have uniform positive lower bounds for elliptic solutions 
of \eqref{Homotopy family}, \eqref{Second boundary data}. Invoking the alternative conditions A4w or A4*w or A3, we then have uniform global second derivative estimates from  \cite{JT-Pogorelov}, (in the first two cases),  and \cite{LT2016}, (in the last case). From the H\" older estimates for second derivatives  \cite {LieTru1986} and the linear theory \cite{GT2001}, we have uniform estimates in the spaces $C^{4,\alpha} (\bar\Omega)$ for $\alpha < 1$, provided $\Omega$ is smooth enough, say $C^5$. Accordingly there exists a bounded open set $\mathcal O$ in $C^{4,\alpha} (\bar\Omega)$ such that the operators $\mathcal F$ and $\mathcal G$ are respectively elliptic and oblique with respect to all $u\in\mathcal O$ and the boundary value problems (3.4), (1.6) have no elliptic solutions in $\partial\mathcal O$. Furthermore the set $\mathcal O$ can be chosen so that $u_0 \in \mathcal O$ and $Tu$ is one-to-one on $\bar \Omega$ for all $u \in \mathcal O$.  This latter property  implies that  condition \eqref{Second boundary data} is in fact equivalent to the oblique condition \eqref {BC} for our solutions,  so that we then conclude the solvability of the boundary value problem \eqref{Homotopy family}, \eqref{Second boundary data}, at $t=1$, from the degree theory for oblique boundary value problems, explicitly from Case (ii) of Theorem 10.23, (with $k=m_1 =1$),  in \cite{FP1993} or from assertions
(a) and (d) of  Corollary 2.1 in \cite{LLN2015}. For this, as well as the uniqueness of our initial solution $u_0$, we also need to observe, from the proof of Lemma \ref{Uniqueness lemma}, that the linearized operator $\mathcal L$, associated with the boundary value problem \eqref{equation MA form}, \eqref{BC} at $t=0$ and $u=u_0$, is one-to-one, whence by virtue of the Schauder theory \cite{GT2001}, $\mathcal L$ is an isomorphism from $C^{4,\alpha} (\bar\Omega)$ to $C^{2,\alpha} (\bar\Omega)\times C^{3,\alpha} (\partial\Omega)$. Note that the treatment of the second order case in \cite{LLN2015},
based on the Dirichlet problem case in \cite {L1989}, is somewhat simpler than the general theory in \cite{FP1993} and the degree constructed there is homotopy invariant whereas in the generality of quasilinear Fredholm operators in  \cite{FP1993}, the sign of the degree may change. 

Finally we complete the proof of Theorem \ref{Th1.1} by sending $\epsilon$ to $0$ in \eqref{Homotopy family} and subsequently in our approximations $\Omega_\epsilon^*$, using again our {\it a priori} solution bounds.

\begin{remark}\label{Remark 3.1}
To prove the existence theorem for classical solutions in the optimal transportation problem \cite{TW2009}, there are two different approaches using the method of continuity. The first approach, (in Section 5 in \cite{TW2009}), is based on  domain deformation, which requires an appropriate foliation which follows  from a global barrier condition, (see (1.21) or (5.7) in \cite{TW2009}). The second approach, (in Section 7 in \cite{TW2009}), is based on a direct construction of uniformly elliptic functions with  approximating target domains, which can be applied without domain deformation. Our proof here, in the more general geometric optics setting, utilizes the second approach in \cite{TW2009} without the domain variation and global barrier. Since we have also obtained the second derivative estimate in \cite{JT-Pogorelov} without the global barrier condition, we can completely avoid the global barrier condition for the existence result, Theorem \ref{Th1.1}. 
\end{remark}

\begin{remark}\label{Remark 3.2}
As remarked in Section \ref{Section 1}, there clearly exist an infinite number of solutions in Theorem \ref{Th1.1}. By inspection of our proof in Sections \ref{Section 2} and \ref{Section 3}, we may also assume that the interval $J_0$ in condition A5 is finite, provided there exists a $g$-affine function $g_0$ satisfying \eqref{2.3} and $Tg_0\in \Omega^*$.  Then we have a more precise result under the hypotheses of Theorem 
\ref{Th1.1}, namely there exists an elliptic solution $u\in C^3(\bar \Omega)$ of the second boundary value problem \eqref{GPJE}, \eqref{Second boundary data} whose graph intersects that of $g_0$.
\end{remark}

\begin{remark}\label{Remark 3.3}
The supplementary conditions, A4w or A4*w or A3, in Theorem \ref{Th1.1} are only used to guarantee global second derivative bounds so it would be interesting to prove such bounds just under conditions A1, A2, A1* and A3w. Such bounds were originally proved in \cite{TW2009} for general Monge-Amp\`ere type equations also under the addition of the global barrier condition; (see equation (3.3) in \cite{TW2009}), which was subsequently removed in \cite{JT-Pogorelov} for optimal transportation equations and, more generally, generated prescribed Jacobian equations, satisfying A4w or A4*w. We remark also that the alternative  duality method proposed in \cite{TW2009}, Theorem 3.2, for the case when $A$ depends only on $p$, is not valid, although the case when $n=2$ still follows directly without using our construction in \cite{JT-Pogorelov}. Similarly the foreshadowed alternative use of duality at the end of \cite{JT-Pogorelov} is only valid for $n=2$. We are grateful to Philippe Delan\"oe for pointing out this problem to us.

\end{remark}

\begin{remark}\label{Remark 3.4}
We  may also consider the second boundary value problem \eqref{Second boundary data} for more general fully nonlinear, augmented Hessian equations of the form,
\begin{equation}\label{augmented Hessian}
F[D^2u- A(\cdot,u,Du)] = B(\cdot,u,Du),
\end{equation}
where $F$ is an increasing function on the positive cone of $n\times n$ symmetric matrices, $A$ given by \eqref{A,B}
is defined through a $C^4$ generating function $g$ and $B$ is a positive function in 
$C^2(\bar\Omega\times \mathbb{R}\times \mathbb{R}^n)$.
The main examples here are the Hessian quotients, $F_{n,k}$ given by 
\begin{equation}\label{Hessian quotient}
F_{n,k} = \frac{\det }{S_k},
\end{equation}
where the $k$-Hessian $S_k(r)$, $0<k<n$, is the sum of the principal $k\times k$ minors of the matrix $r$, which were considered in the optimal transportation case, when $g$ is given by a cost function, by von Nessi \cite {vonNessi2010}. Global second derivative estimates for elliptic solutions, in the case \eqref{Hessian quotient} and more generally were proved by us in  \cite{JT-Oblique II}, Corollary 3.1, under conditions A1, A2, A3w and A4w together with $B$ independent of $p$ satisfying  the monotonicity $B_u\ge 0$ and the existence of an elliptic subsolution $\underline u\in C^2(\bar\Omega)$. For such general operators, we can also assume condition A5 to guarantee the gradient estimate when the range of the solution $u$ lies in $J_0$. If we have appropriate solution estimates, we can readily prove the classical existence result following the steps in Sections \ref{Section 2} and \ref{Section 3} of the current paper. In general, the obstacle for the maximum solution estimate arises from the lack of the structures \eqref{psi} and \eqref{conservation of energy}. However when $B(x,\cdot)(J_0) = (0, \infty)$ for all $x \in \Omega$, we  can obtain the solution estimate by modification of Section 5.2 in \cite{vonNessi2010}, and thus  conclude the classical existence result. 
\end{remark}

%%%%%%%%%%%%%%%%%%%%%%%%%%%%%%%%%%%%%%%%%%%%%%%%%%%%%%%%%%%%%%%%%%%%%%%%%%%

\section{Applications in geometric optics}\label{Section 4}

In this section we treat the application of Theorem \ref{Th1.1} to some  examples in geometric  optics developed  in \cite{JT-Pogorelov} in conjunction with  our global second derivative bounds. In particular we consider  the reflection and refraction of parallel light beams to targets, which are graphs over orthogonal hyperplanes. Our concern, as in the case of flat targets in \cite{LT2016} is with globally smooth solutions and the reader is referred to \cite{T2014, T2014-1, GK2015} for local regularity considerations as well as \cite{GuT2015, K-2014, K-2016} for more general targets. As in \cite{JT-Pogorelov}, we consider parallel beams in $\mathbb{R}^{n+1}$, directed in the direction of $e_{n+1}$, through a domain $\Omega\subset \mathbb{R}^n\times \{0\}$, illuminating  targets  which are graphs over domains $\Omega^*\subset \mathbb{R}^{n}\times \{0\}$. The targets are allowed to be either flat or non-flat.

\subsection{Reflection}\label{Section 4.1} 

Let $\mathcal D$ be a domain in $ \mathbb{R}^n\times\mathbb{R}^n$, containing $\bar\Omega\times\bar\Omega^*$, and consider the generating function:
\begin{equation}\label{reflection}
g(x,y,z) = \Phi(y) - \frac{z}{2} + \frac{1}{2z}|x-y|^2,
\end{equation}
defined for $(x,y) \in\mathcal D$ and $z<0$ where $\Phi$ is a smooth function on $\mathbb{R}^n$.
Here we have replaced $z$ by $-1/z$ in \cite{JT-Pogorelov} to conform with the subsequent refraction examples. From our calculations in Section 4.2(i) of \cite{JT-Pogorelov}, we see that $g$ satisfies conditions A1, A2, A1*, A4w  on $\Gamma \subset \mathbb{R}^n\times \mathbb{R}^n \times \mathbb{R}$, which can be defined through its dual set
$$ \Gamma^* =\{(x,y,u)\in \mathbb{R}^n\times \mathbb{R}^n \times \mathbb{R} | \ (x,y)\in \mathcal D, u\in J(x,y)\},$$
\noindent where
\begin{equation} \label{J}
J(x,y) = (\Phi(y) + (x-y)\cdot D\Phi(y), \infty).
\end{equation}

In the corresponding reflection problem we seek a reflecting surface $\mathcal R$ as a  graph $\{(x,u(x)) | \ x\in \Omega\}$, so that light with intensity $f$ on $\Omega$ is mapped under the reflection mapping $Tu$ to intensity $f^*$ on $\Omega^*$, where $f$ and $f^*$ satisfy the conservation of energy condition \eqref{conservation of energy}. This leads to solving the boundary value problem, \eqref{GPJE}, \eqref{Second boundary data} and the condition $u(x) \in J(x,y)$ for all $x\in\Omega$, $y\in \Omega^*$, arising from \eqref{J}, is equivalent to the reflector $\mathcal R$ lying above the tangent hyperplane to the target $ \mathcal T$ at any $y\in \Omega^*$, which is clearly necessary  when a ray through $x\in\Omega$  illuminates 
$ \mathcal T$ from above at $y\in\Omega^*$. Next from (4.11) in \cite{JT-Pogorelov}, the matrix $A$ is given by
\begin{equation}
A(x,u,p) = \frac{1}{Z(x,u,p)} I
\end{equation}
so that A3w is satisfied if and only if the function $1/Z$ is locally convex in $p$, where $Z$ denotes the dual function. 
Finally we  have, again from the calculations in \cite{JT-Pogorelov}, that condition A5 is satisfied for $J_0=(m_0,\infty)$, where $m_0$ is given by
$$ m_0 = \sup_{ x\in\Omega, y\in\Omega^*} [\Phi(y) + (x-y)\cdot D\Phi(y)],$$ 
and $K_0$ given by
$$ K_0 = \sup_{\Omega^*} (\sqrt{1 + |D\Phi|^2} +|D\Phi|).$$

\subsection{Refraction}\label{Section 4.2} 

We consider refraction from media $I$ to media $I\!I$, through a surface interface  $\mathcal{R} = \{(x,u(x)) | \ x\in \Omega\}$, with respective refraction indices $n_1,n_2 > 0$ and set $\kappa = n_1/n_2$. For $\kappa\ne 1$, we consider now generating functions,
\begin{equation}\label{refraction}
 g(x,y,z) = \Phi(y) - \frac{1}{|\kappa^2 - 1|} \big(\kappa z +\sqrt{z^2 +(\kappa^2-1)|x-y|^2}\big),
\end{equation}
where again $(x,y)\in\mathcal D$, $z > \kappa^\prime|x-y|$ for $0 <\kappa < 1$,
 $> 0$ for $\kappa > 1$, where $\kappa^\prime =\sqrt{|\kappa^2-1|}$, and $\Phi$ is a smooth function on $\mathbb{R}^n$. From our calculations in Section 4.2(ii) of \cite{JT-Pogorelov}, we then obtain that $g$ satisfies conditions A1, A2 and A1* as above, with in place of \eqref{J},
 \begin{equation} 
J(x,y) = (-\infty, \Phi(y) + (x-y)\cdot D\Phi(y))\cap(-\infty, \Phi(y)- \frac{\min\{\kappa,1\}}{\kappa^\prime} |x-y|),
\end{equation}
with condition A4w satisfied for $\kappa <1$ and condition A4*w for $\kappa>1$. Furthermore from (4.19) and (4.22) in
\cite{JT-Pogorelov}, the matrices $A$ are given by
\begin{equation}
A(x,u,p) = [{\rm{sign}}(1-\kappa^2)]\frac{\sqrt{1+(1-\kappa^2)|p|^2}}{ Z(x,u,p)}[I + (1-\kappa^2)p\otimes p],
\end{equation}
so that condition A3w is satisfied if and only if
the function
$$ p \rightarrow \frac{(1-\kappa^2)\sqrt{1+(1-\kappa^2)|p|^2}}{Z(x,u,p)} $$
is locally convex. Finally to complete our hypotheses for the application of Theorem \ref{Th1.1}, we obtain, again 
from the calculations in \cite{JT-Pogorelov}, that condition A5 is satisfied for $J_0=(-\infty,M_0)$, where $M_0$ is given by
$$ M_0 = \inf_{ x\in\Omega, y\in\Omega^*} \min\{\Phi(y) + (x-y)\cdot D\Phi(y), \Phi(y) -\frac{\min\{\kappa,1\}}{\kappa^\prime}(1+\delta)|x-y|\},$$ 
where $\delta >0$ for $\kappa < 1$, $\delta = 0$ if $\kappa > 1$, and $K_0 = 2/ \kappa\kappa^\prime\delta$ for $\kappa<1$, $K_0 = 1/ \kappa^\prime$ for $\kappa >1$. Note that in our refraction model, the constraint $u(x) \in J(x,y)$ for all $x\in\Omega$, $y\in \Omega^*$ implies the reflector $\mathcal R$ lies below the tangent hyperplane to the target 
$ \mathcal T$ at any $y\in \Omega^*$, which is clearly necessary  when a ray through $x\in\Omega$  illuminates 
$ \mathcal T$ from below at $y\in\Omega^*$. 

 \subsection{Flat targets}\label{Section 4.3} 
 
 When the target is  flat, that is $\Phi =$ constant, our models reduce to those considered in
 \cite{LT2016} and condition A3 holds, as is seen readily from the formulae for the dual function $Z$, namely
 $$ Z(x,u,p) = \frac{2(\Phi-u)}{1-|p|^2}, \ u >\Phi, \ |p| < 1$$
 in the case of reflection, and
 $$Z(x,u,p) = \frac{|1-\kappa^2|(\Phi - u)\sqrt{1+ (1-\kappa^2)|p|^2}}{1+\kappa\sqrt{1+ (1-\kappa^2)|p|^2}}, \  u< \Phi, \ (\kappa^2-1)|p|^2 < 1$$
 in the case of refraction. Now applying Theorem \ref{Th1.1}, it follows that the barrier condition (28) in the hypotheses of Theorems 5.1 and 5.2 in \cite{LT2016} can be removed and moreover Remark \ref{Remark 3.2} provides more information about the overall solution set, and is also applicable to Theorem 1.2 in  \cite{LT2016}.

 \subsection{Reverse ellipticity}\label{Section 4.4} 
 From Theorem \ref{Th1.1}, we also obtain classical solutions to the above reflector problems satisfying the reverse ellipticity condition, $D^2u < A (\cdot,u,Du)$. Here again conditions A1, A2, A1*, A4w or A4*w, A5 are satisfied as in Sections \ref{Section 4.1} and \ref{Section 4.2} but we must replace $Z$ by $-Z$ or equivalently convexity by concavity to ensure condition A3w. In these models we have chosen at the outset the ellipticity, or equivalently the  support from below by focusing quadric surfaces corresponding to our $g$-affine functions, in order to embrace the flat target cases in Section \ref{Section 4.3}.

 \vskip 12pt

\textbf{Acknowledgements.} Research partially supported by the National Natural Science Foundation of China (No.11401306) and the Australian Research Council (DP170100929).

\textbf{Conflict of interest.} There are no conflicts of interest.

\end{document}